\documentclass{{scrartcl}}
\usepackage{amsmath}
\usepackage{bbm}
\usepackage{amssymb}
\usepackage{amsthm}

\newtheorem{lemma}{Lemma}[section]
\newtheorem{theorem}{Theorem}[section]

\theoremstyle{remark}
\newtheorem{remark}{Remark}[section]
\newcommand{\C}{\mathbb{C}}
\newcommand{\N}{\mathbb{N}}
\newcommand{\Z}{\mathbb{Z}}
\newcommand{\R}{\mathbb{R}}

\numberwithin{equation}{section}

\begin{document} 
\title{On the asymptotic normality of the Legendre-Stirling numbers of the second kind}
\author{Wolfgang Gawronski
  \thanks{Department of Mathematics, University of Trier, Germany. E-mail: gawron@uni-trier.de}\\
	Lance L. Littlejohn
	\thanks{Department of Mathematics, Baylor University, TX, USA. E-mail: Lance\_Littlejohn@baylor.edu}\\
	Thorsten Neuschel
	\thanks{Department of Mathematics, KU Leuven, Belgium. E-mail: Thorsten.Neuschel@wis.kuleuven.be}}

 \date{\today}

\maketitle






\maketitle

\begin{abstract}For the Legendre-Stirling numbers of the second kind asymptotic formulae are derived in terms of a local central limit theorem. Thereby, supplements of the recently published asymptotic analysis of the Chebyshev-Stirling numbers are established. Moreover, we provide results on the asymptotic normality and unimodality for modified Legendre-Stirling numbers.
\end{abstract}

\paragraph{Keywords} Jacobi-Stirling numbers, Legendre-Stirling numbers, local central limit theorem, uniform asymptotics, asymptotic normality, unimodality.

\paragraph{Mathematics Subject Classification (2010)} 05A16, 60F05  
\section{Introduction and summary} \label{intro}
The main objects of our investigations are the Legendre-Stirling numbers, a special case of the Jacobi-Stirling numbers both of the second kind.  Following the recent literature, e.g., \cite{3}, \cite{9}, \cite{15}, we denote the latter numbers by the curly bracket symbol \({n\brace j}_\gamma\), $n,j$  being non-negative integers, and $\gamma$ is a fixed non-negative parameter. A formal definition of these numbers can be given through the triangular recurrence relation \cite{3}, \cite{15}
\begin{equation}\label{eq:1.1}{n\brace j}_\gamma = {n-1\brace j-1}_\gamma +j(j+2\gamma-1){n-1\brace j}_\gamma,\quad\quad n,j\in \mathbb{N}=\{1, 2,\ldots\},\end{equation}
\[{n\brace 0}_\gamma = \delta_{n,0},\quad\quad{0\brace j}_\gamma = \delta_{j,0},~~~~~~~n,j\in \mathbb{N}_0=\{0, 1, 2,\ldots\}.\]
We note that they were originally discovered in the left-definite spectral analysis of integral powers of the second order Jacobi differential operator
\[
\ell [y] (x) := \frac{-1}{(1-x)^\alpha (1+x)^\beta} \, \Big( (1-x)^{\alpha+1} (1+\gamma)^{\beta+1} y'(x)\Big)'
\]
where $\alpha, \beta > -1$ are constants and $x \in (-1,1)$. Indeed, the numbers $\{^n_j\}_\gamma$ occur in the explicit representation formula for the powers
\begin{equation} \label{eq:1.2}
\ell^n [y] (x) = \frac{1}{(1-x)^\alpha (1+\gamma)^\beta} \, \sum_{j=1}^n (-1)^j {n\brace j}_\gamma
\Big( (1-x)^{\alpha+j} (1+x)^{\beta+j} y^{(j)} (x)\Big)^{(j)}\, ,
\end{equation}
where $n\in\N$ and $\gamma = (\alpha + \beta + 2)/2$. For precise information in this context the reader is refered to the pertinent literature \cite{11}, \cite{12}, \cite{20}. The Legendre differential operator corresponds to the case  $\alpha = \beta = 0$, i.e., $\gamma = 1$, and hence the numbers \({n\brace j}_1\) are called Legendre-Stirling numbers of the second kind. Besides the already mentioned original field of differential equations during the past decade the Jacobi-Stirling numbers received considerable attention especially in combinatorics and graph theory, see, e.g.,
\cite{1}, \cite{2}, \cite{3}, \cite{6}, \cite{7}, \cite{9}, \cite{15}, \cite{16}, \cite{17},  \cite{21}, \cite{22}, \cite{23}. Among the \({n\brace j}_\gamma\)'s the Legendre-Stirling numbers \({n\brace j}_1\) were the first ones which have been examined in detail \cite{1}, \cite{2}, \cite{9}. The Jacobi-Stirling numbers and, in particular, the Legendre-Stirling numbers share many similar properties to those of the classical Stirling numbers of the second kind. In fact, analogous to (\ref{eq:1.2}), the Stirling numbers of the second kind are the coeffcients of integral powers of the classical Laguerre differential expression; see \cite{12} and \cite{20} for further details.

For the unique solution of the recurrence \eqref{eq:1.1} the following explicit formula is known \cite{3}, \cite{12}, \cite{15}, \cite{16}
\begin{equation} \label{eq:1.3}
{n\brace j}_\gamma =\sum_{r=0}^j (-1)^{r+j} \, 
\frac{(2r+2\gamma -1) \Gamma (r+2\gamma-1)(r(r+2\gamma-1))^n}{r! (j-r)! \Gamma (j+r+2\gamma)}\, ,
\end{equation}
where $n,j \in\N_0$ and $\gamma > 0$. As an immediate consequence we obtain the asymptotic statement
\begin{equation} \label{eq:1.4}
{n\brace j}_\gamma \sim
\frac{\Gamma (j+2\gamma-1)}{j! \Gamma (2j+2\gamma-1)} \,
\big( j(j+2\gamma-1)\big)^n, ~~ \text{as } n\to\infty,
\end{equation}
provided that $j \geq 1$ is \textit{fixed}, where throughout the symbol $\sim$ means that the ratio of both sides in 
\eqref{eq:1.4} tends to $1$ as customary. As known from the literature on many double sequences being significant in combinatorics much more interesting than the ``pointwise asymptotics'' in \eqref{eq:1.4} are asymptotic properties of \({n\brace j}_\gamma\), as $n\to\infty$, holding uniformly with respect to $j \in\Z$, where $\Z$ denotes the set of integers as usual. The only case of the Jacobi-Stirling numbers which has been treated under this aspect is that one for which  $\gamma = 1/2$. In the recent article \cite{15} for the so-called Chebyshev-Stirling numbers \({n\brace j}_{1/2}\) the authors have derived asymptotic approximations for these numbers in terms of a local central limit theorem given by (see \cite[Theorem 6.2]{15})
\begin{equation} \label{eq:1.5}
\frac{\sqrt{5b_n'} (2j)! \omega^{2n+1}}{2(2n)!} \,
{n\brace j}_{1/2} = \frac{1}{\sqrt{2\pi}}~ e^{-x^2/2}~
\left( 1 + \frac{c_n' (x^3 - 3x)}{6\sqrt{n}}\right) 
+ o\, \left( \frac{1}{\sqrt{n}}\right),
\end{equation}
as $n\to\infty$, where $x = (j-a'_n)/\sqrt{b'_n}$ and explicitly given elementary quantities $a'_n, b'_n, c'_n, \omega$. In particular the remainder term in \eqref{eq:1.5} holds uniformly with respect to $j\in\Z$. The proof for \eqref{eq:1.5}, which relies on a series of special properties of the Chebyshev-Stirling numbers does not work for the Jacobi-Stirling numbers in general (see also the remarks following Theorem 2.9 below). However, by a more elaborate analysis than in \cite{15} it turns out that asymptotics being similar to \eqref{eq:1.5} can be derived for the Legendre-Stirling numbers as well. More precisely, the main result of this paper is given by
\begin{equation} \label{eq:1.6}
\frac{\sqrt{b_n} (2j)! \omega^{2n+1}}{(2n)!} \,
{n\brace j}_1 = \frac{1}{\sqrt{2\pi}}~ e^{-x^2/2}~
+ o(1),
\end{equation}
as $n\to\infty$, where $x = (j-a_n)/\sqrt{b_n}$ and again elementary quantities $a_n, b_n,\omega$ given explicitly in Lemma 4.1. As above the remainder term holds uniformly with respect to $j\in\Z$. An immediate consequence of \eqref{eq:1.6} is that the numbers $(2j)! {n\brace j}_1$ are asymptotically normal (Theorem 4.3).

A short discussion of approximations of the kind \eqref{eq:1.5}, \eqref{eq:1.6} is given in Sections 3 and 4 below. In establishing \eqref{eq:1.6} the main tools are taken from the central limit theory of probability (see Section 3) and from the analysis of the basic case ${n\brace j}_{1/2}$ treated in \cite{15} which has to be extended by some lemmata for the Legendre-Stirling numbers ${n\brace j}_1$ presented in the subsequent Section 2. A related approach to asymptotics occasionally has been applied to various sequences of special numbers in the literature, see, e.g.,
\cite{4}, \cite{5}, \cite{8}, \cite{13}, \cite{14}, \cite{15}, \cite{18}, \cite{19}, \cite{27}, \cite{28}.

\section{Auxiliary results} \label{auxi}

In this section we collect and prove some analytic facts which are basic for our main result in Section 4. To begin with, we note the special case of the triangular recurrence relation \eqref{eq:1.1} for the Legendre-Stirling case given by
\begin{equation}\label{eq:2.1}{n\brace j}_1 = {n-1\brace j-1}_1 +j(j+1){n-1\brace j}_1 \quad\quad\quad n,j\in \mathbb{N},\end{equation}
\[{n\brace 0}_\gamma = \delta_{n,0},\quad\quad{0\brace j}_\gamma = \delta_{j,0},~~~~~~~n,j\in \mathbb{N}_0,\]
and formula \eqref{eq:1.3} reduces to (see [2,(1.3)])
\begin{equation} \label{eq:2.2}
{n\brace j}_1  = \sum_{r=0}^j (-1)^{r+j} \, \frac{(2r+1)(r(r+1))^n}{(j-r)! (j+r+1)!} \, .
\end{equation}
Next, from \cite[(1.3)]{2} we take the representation
\begin{equation} \label{eq:2.3}
{n\brace j}_1 = \frac{1}{(2j)!} \, \sum_{\nu=0}^{2j} (-1)^\nu {2j \choose \nu}
\big( (j-\nu)(j+1-\nu)\big)^n\, ,
\end{equation}
$n,j\in\N_0$, and for the Chebyshev-Stirling numbers from \cite[Lemma 3.2]{15} we get
\begin{equation} \label{eq:2.4}
{n\brace j}_{1/2} = \frac{1}{(2j)!} \, \sum_{r=0}^{2j} (-1)^r {2j \choose r} (j-r)^{2n}\, ,
\end{equation}
$n,j \in\N_0$. Now, these representations imply the following connection formulae relating the Legendre-Stirling numbers
with the Chebyshev-Stirling numbers given by
\begin{lemma} \label{lem2.1}
If $k,j\in \N_0$, then we have
\begin{align} 
\label{eq:2.5} 
{2k+1\brace j}_1 & = \sum_{\mu=0}^k \, {2k+1 \choose 2\mu+1} \,
{k+\mu+1\brace j}_{1/2}\, , \\
\label{eq:2.6} {2k\brace j}_1 & = \sum_{\mu=0}^k \, {2k \choose 2\mu} \, 
{k+\mu\brace j}_{1/2}\, .
\end{align}
\end{lemma}
\begin{proof}
Using  \eqref{eq:2.3} we obtain
\begin{equation} \label{eq:2.7}
{n\brace j}_1 = \sum_{\mu=0}^n \, {n \choose \mu} \, \frac{1}{(2j)!} \, \sum_{\nu=0}^{2j} (-1)^\nu \,
{2j \choose \nu} \, (j-\nu)^{n+\mu}\, .
\end{equation}
We observe that the inner sum is zero provided that $n + \mu = 2\ell + 1$ is odd; indeed we have
\[
\sum_{\nu=0}^{2j} (-1)^\nu {2j \choose \nu}\, (j-\nu)^{2\ell+1} =
\sum_{\nu=0}^{j-1} \ldots + \sum_{\nu=j+1}^{2j} \ldots = 0
\]
by making the index shift $\nu \mapsto 2j-\nu$ in the second sum. Thus, combining \eqref{eq:2.4} and \eqref{eq:2.7} we get
\begin{align*}
{2k+1\brace j}_1
 =& \sum_{\mu=0}^k\, {2k+1 \choose 2\mu+1} \, \frac{1}{(2j)!} \, \sum_{\nu=0}^{2j} (-1)^\nu \, {2j\choose\nu}\, (j-\nu)^{2k+1+2\mu+1}\\ 
 =& \sum_{\mu=0}^k \, {2k+1\choose 2\mu+1}\, {k+\mu+1\brace j}_{1/2},
\end{align*}
which establishes \eqref{eq:2.5} and
\begin{align*}
{2k\brace j}_1
 =& \sum_{\mu=0}^k\, {2k \choose 2\mu} \, \frac{1}{(2j)!} \, \sum_{\nu=0}^{2j} (-1)^\nu \, {2j\choose\nu}\, (j-\nu)^{2k+2\mu}\\ 
 =& \sum_{\mu=0}^k \, {2k\choose 2\mu}\, {k+\mu\brace j}_{1/2},
\end{align*}
which proves \eqref{eq:2.6}.
\end{proof}
Next, following \cite[(15)]{15} we consider the horizontal generating function of the modified Chebyshev-Stirling numbers
$(2j)! {n\brace j}_{1/2}$ given by
\begin{equation} \label{eq:2.8}
L_n (s) := \sum_{j=0}^n (2j)!
{n\brace j}_{1/2} s^j, \qquad s\in\C,
\end{equation}
and the corresponding function for the modified Legendre-Stirling numbers $(2j)!  {n\brace j}_1$ defined through
\begin{equation} \label{eq:2.9}
M_n (s) := \sum_{j=0}^n (2j)! 
{n\brace j}_1 s^j, \qquad s\in\C.
\end{equation}
As an immediate consequence of Lemma \ref{lem2.1} we conclude in
\begin{lemma} \label{lem2.2}
If $k\in\N_0, s\in \C$, then we have
\begin{align} 
\label{eq:2.10} M_{2k+1} (s) & = \sum_{\mu=0}^k {2k+1\choose 2\mu+1} \, L_{k+\mu+1} (s), \\
\label{eq:2.11} M_{2k} (s)   & = \sum_{\mu=0}^k {2k \choose 2\mu} \, L_{k+\mu} (s).
\end{align}
\end{lemma}
In \cite[Lemma 3.3, iii)]{15} for the polynomials $L_n$ the following representation by means of an Eisenstein series turned out to be a very useful analytic tool (see, e.g., \cite{10}, \cite[p. 234]{26})
\begin{equation} \label{eq:2.12}
L_n \left( \frac{1}{2(\cosh w - 1)}\right) = (2n)! \, \frac{2(\cosh w - 1)}{\sinh w} \,
\sum_{m=-\infty}^\infty \, \frac{1}{(w + 2\pi im)^{2n+1}}\, ,
\end{equation}
where $n\in\N$ and $w \in \C\setminus \{2\pi im\, |\,m\in\Z\}$. In order to derive a similar formula for $M_n$ we introduce the Laplace integral
\begin{align} \label{eq:2.13}
I_{r,n} (z) := \int\limits_0^\infty e^{-\xi z} \xi^r \Big( (\xi+1)^n + (\xi-1)^n\Big)\, d\xi\, ,
\end{align}
where $n,r \in \N_0$, $\textrm{Re}\, z > 0$, and prove
\begin{lemma} \label{lem2.3}
If $n\in\N$, then for $w\in\C$ with ${\textrm Re}\, w > 0$ we have
\begin{equation} \label{eq:2.14}
M_n \left( \frac{1}{2(\cosh w - 1)}\right) = \frac{\cosh w - 1}{\sinh w} \, \sum_{m=-\infty}^\infty
I_{n,n} (w + 2\pi im).
\end{equation}
\end{lemma}
\begin{proof}
We recall the well-known binomial identities
\begin{align}
\label{eq:2.15} \sum_{\mu=0}^k {2k+1 \choose 2\mu+1}\, x^{2\mu+1}
& = \frac{1}{2}\, \Big( (x+1)^{2k+1} + (x-1)^{2k+1}\Big), \\ 
\label{eq:2.16} \sum_{\mu=0}^k {2k \choose 2\mu} \, x^{2\mu} 
& = \frac{1}{2}\, \Big( (x+1)^{2k} + (x-1)^{2k}\Big),
\end{align} 
where $k\in\N_0$ and $x\in\C$. First, we consider $n = 2k+1$ to be odd. Then, using \eqref{eq:2.10}, \eqref{eq:2.12}, \eqref{eq:2.13}, and \eqref{eq:2.15}, we obtain
\begin{eqnarray*}
\lefteqn{M_{2k+1} \left( \frac{1}{2(\cosh w - 1)}\right)} \\ \\
&& = \sum_{\mu=0}^k {2k+1 \choose 2\mu+1} \, \big( 2(k+\mu+1)\big)! ~\, \frac{2(\cosh w - 1)}{\sinh w} \,
     \sum_{m=-\infty}^\infty \, \frac{1}{(w+2\pi im)^{2(k+\mu+1)+1}} \\ \\
&& = \frac{2(\cosh w - 1)}{\sinh w} \, \sum_{m=-\infty}^\infty \frac{1}{(w+2\pi im)^{2k+2}}
     \int\limits_0^\infty e^{-t} t^{2k+1} \sum_{\mu=0}^k \, {2k+1 \choose 2\mu+1}\,
     \left( \frac{t}{w + 2\pi im}\right)^{2\mu+1} dt \\
&& = \frac{\cosh w - 1}{\sinh w} \, \sum_{m=-\infty}^\infty \int\limits_0^\infty e^{-\xi (w+2\pi im)} \xi^{2k+1}
     \Big( (\xi+1)^{2k+1} + (\xi-1)^{2k+1}\Big)\, d\xi \\ \\
&& = \frac{\cosh w - 1}{\sinh w} \, \sum_{m=-\infty}^\infty I_{2k+1,2k+1} (w + 2\pi im).          
\end{eqnarray*}
This establishes \eqref{eq:2.14} for odd $n = 2k+1$. If $n = 2k$ is even, then we combine \eqref{eq:2.11}, \eqref{eq:2.12}, 
\eqref{eq:2.13} and \eqref{eq:2.16}. We omit the calculations, since they are very similar to those above.
\end{proof}
From the definition in \eqref{eq:2.13} instantly we get a formula for the derivatives in
\begin{lemma} \label{lem2.4}
If $n,r,\nu\in\N_0$, then for $\textrm{Re}\,z > 0$ we have
\begin{equation} \label{eq:2.17}
I_{r,n}^{(\nu)} (z) = (-1)^\nu\, I_{r+\nu,n} (z)\, .
\end{equation}
\end{lemma}
Next, for our modified Eisenstein series we derive a useful error estimate (c.f. Lemma 5.1 in \cite{15}).
\begin{lemma} \label{lem2.5}
If $n,r\in\N_0, r\geq 2$, and $w > 0$, then we have
\begin{equation} \label{eq:2.18}
\sum_{m=-\infty}^\infty I_{r,n} (w+2\pi im) = I_{r,n} (w)
\left( 1 + {\mathcal{O}} \left( \left( \frac{w}{4\pi}\right)^{(r+1)/2}\right)\right)\, ,
\end{equation}
the ${\mathcal{O}}$-term holding uniformly with respect to $w > 0$ and being independent of $n$ and $r$.
\end{lemma}
\begin{proof}
Writing
\begin{align*}
\sum_{m=-\infty}^\infty I_{r,n} (w+2\pi im)
& =   I_{r,n} (w) \left( 1 + \sum_{m\not= 0} \, \frac{I_{r,n} (w+2\pi im)}{I_{r,n} (w)}\right) \\*[0.3cm]
& =:  I_{r,n} (w) \big( 1 + R_{r,n} (w)\big),
\end{align*}
we have
\begin{align*}
\big| R_{r,n} (w)\big|
& = \left|\sum\limits_{m\not= 0} \left( \frac{w}{w+2\pi im}\right)^{r+1}\,
    \frac{\int\limits_0^\infty e^{-t} t^r \Big( \Big( \frac{t}{w+2\pi im} + 1\Big)^n +
     \Big( \frac{t}{w+2\pi im} - 1\Big)^n\Big)\, dt}
     {\int\limits_0^\infty e^{-t} t^r \Big( \Big( \frac{t}{w} + 1\Big)^n + \Big( \frac{t}{w} - 1\Big)^n\Big)\, dt} 
     \right|\, . 
\end{align*}
In view of the inequality (use \eqref{eq:2.15}, \eqref{eq:2.16})
\begin{eqnarray*}
\left| \int\limits_0^\infty e^{-t} t^r 
        \left( \left( \frac{t}{w+2\pi im} + 1\right)^n +
               \left( \frac{t}{w+2\pi im} + 1\right)^n\right)\, dt \right| \\
\leq \int\limits_0^\infty e^{-t} t^r \left( \left( \frac{t}{w} + 1\right)^n +
               \left( \frac{t}{w} - 1\right)^n\right)\, dt
\end{eqnarray*}
we get
\[
\big| R_{r,n} (w)\big| \leq
2 \sum_{m=1}^\infty\, \frac{w^{r+1}}{(w^2 + 4\pi^2 m^2)^{(r+1)/2}} \leq
2 \zeta\, \left( \frac{r+1}{2}\right) \left( \frac{w}{4\pi}\right)^{(r+1)/2}\, ,
\]
which implies \eqref{eq:2.18}.
\end{proof}
An important asymptotic approximation for our purpose below is provided by
\begin{lemma} \label{lem2.6}
For fixed $\nu\in\N_0$ and $z > 0$ we have
\begin{equation} \label{eq:2.19}
I_{n+\nu,n} (z) = \frac{n! (n+\nu)!}{\sqrt{\pi n}}\, \left( \frac{2}{z}\right)^{2n+\nu+1}
\left( b(z) + \frac{b_\nu (z)}{n} + {\mathcal{O}} \left( \frac{1}{n^2}\right)\right)
\end{equation}
as $n\to\infty$, where
\begin{align} 
\label{eq:2.20} b (z) & = \cosh \, \frac{z}{2}\, , \\*[0.3cm]
\label{eq:2.21} b_\nu (z) & = - \frac{1}{8} \left( \frac{z^2}{2}\, \cosh \, \frac{z}{2} + 2z \nu \sinh \, \frac{z}{2} +
            (2\nu^2 + 2\nu+1)\, \cosh \, \frac{z}{2}\right)\, .
\end{align}
\end{lemma}
\begin{proof}
From \eqref{eq:2.13} we obtain
\begin{align*} 
I_{n+\nu,n} (z)
& = (-1)^{n+\nu} \left( \frac{d}{dz}\right)^{n+\nu} \int\limits_0^\infty e^{-z\xi}
\Big( (\xi+1)^n + (\xi-1)^n\Big)\, d\xi \\*[0.2cm]
 &= (-1)^{n+\nu} \left( \frac{d}{dz}\right)^{n+\nu} 
    \left\{ e^z (-1)^n \left( \frac{d}{dz}\right)^n\, \frac{e^{-z}}{z} + e^{-z} (-1)^n
    \left( \frac{d}{dz}\right)^n\, \frac{e^z}{z}\right\} \\[0.4cm]
 = (-1)^{\nu}&\left( \frac{d}{dz}\right)^{n+\nu}
    \left\{ e^z \left( \frac{d}{dz}\right)^n \, \frac{1}{2\pi i} \int \, \frac{e^{-\tau}}{\tau}\, \frac{d\tau}{\tau-z}
    + e^{-z} \left( \frac{d}{dz}\right)^n\, \frac{1}{2\pi i} \int\, \frac{e^\tau}{\tau} \, \frac{d\tau}{\tau-z}\right\} \, ,
\end{align*}
where the integration is performed on a small circle around $\tau = z$ with positive orientation. Further, we get $(\tau - z = t)$
\begin{align}  \label{eq:2.22}
I_{n+\nu,n} (z)
& = \frac{(-1)^\nu n!}{2\pi i} \, \left( \frac{d}{dz}\right)^{n+\nu} \int
    (e^{z-\tau} + e^{\tau-z}) \, \frac{d\tau}{\tau(\tau-z)^{n+1}} \nonumber \\[0.3cm] 
& = \frac{(-1)^n n! (n+\nu)!}{\pi i} \int \,\frac{\cosh t}{t^{n+1} (t+z)^{n+\nu+1}}\, dt\, ,
\end{align}
now the integration being performed on a small circle around $t = 0$ with positive orientation. Next, we will evaluate the integral by a saddle point approximation (see \cite[Chapter 4.7, $\lambda = 1, \mu = 2$]{24}). To this end, we consider the integral
\begin{align} \label{eq:2.23}
\int e^{-n p(t)} q(t)\, dt
\end{align}
with a contour of integration as above and
\begin{align} \label{eq:2.24}
p(t) := \log t (t + z), \qquad
q(t) := \frac{\cosh t}{t(t+z)^{\nu+1}} \, ,
\end{align}
the branch of the logarithm being real when $t\in (0,\infty)$. From a simple calculation we get that $t_0 = -z/2$ is the only candidate for a saddle point and $p''(t_0) = -8/z^2$. Thus, we choose the oriented circle, given by
\[
t(\varphi) = \frac{z}{2} \, e^{i\varphi}, \qquad 0 \leq \varphi \leq 2\pi, 
\]
as a new path of integration. Observe that $t(\pi) = t_0$. Since for every $\varphi \in [0,2\pi]$
\begin{align*}
\textrm{Re}\, \big(  p(t(\varphi)) - p(t_0)\big)
& = \log \left|\frac{z}{2}\, e^{i\varphi} \left( \frac{z}{2}\, e^{i\varphi} + z\right)\right| -
    \log \left( \frac{z}{2}\right)^2 \\*[0.3cm]
& = \log |e^{i\varphi} + 2| = \frac{1}{2} \, \log \big( 1 + 4 (1 + \cos\varphi)\big) \geq 0\, ,       
\end{align*}
with equality, if and only if $\varphi = \pi$, we may apply the saddle point approximation as described, e.g., in \cite[Chapters 4.6, 4.7]{24} and finally we obtain the asymptotic expansion
\begin{align} \label{eq:2.25}
\int e^{-np(t)} q(t) dt \approx \frac{2}{\sqrt{n}}\, e^{-np(t_0)} \sum_{s=0}^\infty \Gamma \left( s + \frac{1}{2}\right)\,
\frac{a_{2s}}{n^s},
\end{align}
as $n\to\infty$, with constants $a_{2s}$, in particular
\begin{align} \label{eq:2.26}
a_0 = \frac{q}{(2p'')^{1/2}}\, , \quad
a_2 = \left\{ 2q' - \frac{2p''' q'}{p''} + \left( \frac{5p'''^2}{6p''^2} - \frac{p^{(4)}}{2p''}\right) q\right\}
      \, \frac{1}{(2p'')^{3/2}}
\end{align}
with derivatives taken $t = t_0$. Since
\[
\lim_{\varphi\to\pi +} \arg \big( t(\varphi) - t(\pi)\big) = \frac{3\pi}{2}\, ,
\]
in forming the non integral powers of $p'' (t_0)$ we have to choose that branch of 
$\psi = \arg p'' (t_0) = \arg (-8/z^2)$ satisfying $|\psi + 3\pi| \leq \pi/2$ which gives $\psi = -3\pi$. Thus we have 
$\big( 2p''(t_0)\big)^{1/2} = 4i/z$ and cumbersome but straightforward computations lead to \eqref{eq:2.19} -- \eqref{eq:2.21} by combining \eqref{eq:2.22} -- \eqref{eq:2.26}, which completes the proof.
\end{proof}
\begin{remark} \label{rem2.7}
The proof of Lemma \ref{lem2.6} (see \eqref{eq:2.25}) shows that $I_{n+\nu,n} (z)$ possesses a complete asymptotic expansion as $n\to\infty$. We only computed the first two coefficients. The saddle point analysis as described in \cite{24} shows how to calculate more coefficients in \eqref{eq:2.25}, however we do not need them explicitly.
\end{remark}
Another basic fact we require below is the negativity of the zeros of $M_n$. Unfortunately we do not get this information as quick as for the polynomials $L_n$ in \cite[Lemma 3.4]{15}. We start from the equations \eqref{eq:2.1}. For the modified Legendre-Stirling numbers
\begin{equation} \label{eq:2.27}
{n\brace j}_1^\ast := (2j)! {n\brace j}_1  ,
\end{equation}
$n,j\in\N_0$, we immediately get the recurrence
\begin{equation}\label{eq:2.28}{n\brace j}_1^\ast = 2j(2j-1){n-1\brace j-1}_1^\ast +j(j+1){n-1\brace j}_1^\ast \quad\quad\quad n,j\in \mathbb{N},\end{equation}
\[{n\brace 0}_1^\ast = \delta_{n,0},\quad\quad{0\brace j}_1^\ast = \delta_{j,0},~~~~~~~n,j\in \mathbb{N}_0.\]
Now (see \eqref{eq:2.9}) for
\[
M_n (s) = \sum_{j=0}^n 
{n\brace j}_1^\ast s^j
\]
an easy calculation implies that
\begin{equation} \label{eq:2.29}
M_n (s) = s \big\{ 2M_{n-1} (s) + (10s+2) M_{n-1}' (s) + s (4s+1) M_{n-1}'' (s)\big\}, ~~ n\geq 1,
\end{equation}
with $M_0 (s) = 1$. Hence we obtain
\begin{equation} \label{eq:2.30}
M_1 (s) = 2s, \quad M_2(s) = 4s(6s+1), \quad M_3 (s) = 8s (90s^2 + 24s+1)\,.
\end{equation}
\begin{theorem} \label{theo2.8}
For every $n \geq 1$ all zeros of $M_n$ are real, simple, and they are located in the interval
$\big( - \frac{1}{4},0\big]$.
\end{theorem}
\begin{proof}
We proceed by induction with respect to $n$ (c.f. \cite[Theorem 5.7]{2}). To begin with, by \eqref{eq:2.30}, we note that the theorem is true for $n = 1,2,3$. Next, we assume that the assertion holds for $M_{n-1}$, if $n \geq 4$. Thus, the zeros $s_{n-1,\nu}$ of $M_{n-1}$ $(M_{n-1} (0) = 0$, by \eqref{eq:2.29}) satisfy 
\begin{equation} \label{eq:2.31}
- \frac{1}{4} < s_{n-1,n-1} < \ldots < s_{n-1,2} < s_{n-1,1} = 0\, .
\end{equation}
Consequently, $M_{n-1}$ has $n-2$ relative extreme points, $t_{n-1,\nu}$ say, $\nu = 1,\ldots,n-2$, such that
\begin{equation} \label{eq:2.32}
- \frac{1}{4} < t_{n-1,n-2} < t_{n-1,n-3} < \ldots < t_{n-1,1} < 0
\end{equation}
and
\begin{equation} \label{eq:2.33}
\textrm{sign}\, M_{n-1} (t_{n-1,\nu}) = (-1)^\nu, \qquad \textrm{sign}\, M_{n-1}'' (t_{n-1,\nu}) = (-1)^{\nu+1}\, ,
\end{equation}
$\nu = 1,\ldots,n-2$ (observe that all coefficients of $M_{n-1}$ are positive). Since $M_{n-1}' (t_{n-1,\nu}) = 0$, $\nu = 1,\ldots,n-2$, from \eqref{eq:2.29} we get
\[
M_n (t_{n-1,\nu}) = t_{n-1,\nu} \Big\{ 2M_{n-1} (t_{n-1,\nu}) + t_{n-1,\nu} (4t_{n-1,\nu}+1) M_{n-1}'' (t_{n-1,\nu})\Big\}
\]
and hence $(t_{n-1,\nu} < 0, 4t_{n-1,\nu} + 1 > 0)$
\begin{equation} \label{eq:2.34}
\textrm{sign}\,M_n (t_{n-1,\nu}) = (-1)^{\nu+1}, \quad \nu = 1,\ldots,n-2\, .
\end{equation}
Thus, there exist $n - 3$ zeros, $s_{n,\nu}$ say, of $M_n$, $\nu = 3,\ldots,n-1$, such that
\[
- \frac{1}{4} < t_{n-1, n-2} < s_{n,n-1} < t_{n-1,n-3} < s_{n,n-2} < \ldots <
t_{n-1,2} < s_{n,3} < t_{n-1,1} < 0.
\]
Since $M_n (0) = 0, M_n' (0) = {n\brace 1}_1^\ast = 2 {n\brace 1}_1 > 0$ and
$\textrm{sign}\,M_n (t_{n-1,1}) = 1$, there is an additional zero $s_{n,2} \in (t_{n-1,1}, 0)$ of $M_n$. So far we have proved that $M_n$ has $n-1$ zeros $s_{n,\nu}$, satisfying
\[
- \frac{1}{4} < t_{n-1, n-2} < s_{n,n-1} < \ldots < t_{n-1,1} < s_{n,2} < s_{n,1} = 0.
\]
Finally, from \eqref{eq:2.34} we conclude
\begin{equation} \label{eq:2.35}
\textrm{sign}\, M_n (t_{n-1,n-2}) = (-1)^{n-1}
\end{equation}
and further from \eqref{eq:2.29} we infer
\begin{equation} \label{eq:2.36}
\textrm{sign}\,M_n \left( - \frac{1}{4}\right) = \textrm{sign}\, \left( - \frac{1}{4}\right)
\textrm{sign}\, \left\{ 2M_{n-1} \left( - \frac{1}{4}\right) - \frac{1}{2} \, M_{n-1}' 
\left( - \frac{1}{4}\right) \right\} = (-1)^n\, ,
\end{equation}
which is true in view of \eqref{eq:2.31}, \eqref{eq:2.32} and because $M_{n-1},$ $M_{n-1}'$ have non-negative coefficients only. Concluding, \eqref{eq:2.35} and \eqref{eq:2.36} ensure that $M_n$ has an additional zero 
$s_{n,n} \in \big( - \frac{1}{4}, t_{n-1,n-2}\big)$. Summarizing, Theorem \ref{theo2.8} is proved.
\end{proof}
Using a well-known criterion for the unimodality of a sequence of real numbers, for example \cite[Theorem B, p. 270]{8}, from Theorem \ref{theo2.8} we immediately obtain a supplement of the unimodality property of the Legendre-Stirling numbers \cite[Theorem 5.9]{2}.
\begin{theorem} \label{theo2.9}
If $n\geq 3$, the numbers $(2j)!{n\brace j}_1$, $0\leq j \leq n$, are unimodal with either a peak or a plateau of two points.
\end{theorem}
For the Jacobi-Stirling numbers ${n\brace j}_\gamma$ in general the polynomials
\[
M_n^\gamma (s) = \sum_{j=0}^n (2j)!\, {n\brace j}_\gamma \, s^j
\]
satisfy a recurrence being similar to \eqref{eq:2.29}. However, e.g. $M_3^\gamma$ has non real zeros, when $\gamma$ is sufficiently large. Thus, there is no analogue of Theorem \ref{theo2.8} and the arguments below cannot be applied for general $\gamma$.

\section{Tools from probability theory}
In this section we briefly provide some notations from probability and present a general local central limit theorem which turns out to be a basic tool for our main result below. In order to keep this article self-contained we summarize essential topics from section 4 of \cite{15}, see also section 3 of \cite{14}. This probalistic point of view occasionally has been used in the literature for computing asymptotics, e.g. \cite{4}, \cite{5}, \cite{8}, \cite{13}, \cite{14}, \cite{15}, \cite{18}, \cite{27}, \cite{28}.

Motivated by the generating polynomials $M_n$ in \eqref{eq:2.9} with real and non-positive zeros only (Theorem 2.8) we consider the polynomials
\begin{equation} \label{eq:3.1}
A_n (s) := \sum_{j=0}^n \alpha_{nj} s^j = \alpha_{nn} \prod_{\nu=1}^n (s+x_{n\nu})
\end{equation}
with $x_{n\nu} \geq 0, \nu = 1,\ldots,n$, and ask for asymptotics of the coefficients $\alpha_{nj}$, as $n\to\infty$, uniformly in $j$. Now the polynomials
\begin{equation} \label{eq:3.2}
\frac{A_n (s)}{A_n(1)} = \prod_{\nu=1}^n\, (p_{n\nu} s + 1 - p_{n\nu})\, ,
\end{equation}
where $p_{n\nu} := 1/(1+x_{n\nu})$, may be regarded as the generating functions of the row sums
\[
S_n = \sum_{\nu=1}^n X_{n\nu}
\]
of a triangular array of Bernoulli random variables
\[
(X_{n\nu})_{1\leq \nu \leq n}\, .
\]
Due to the factorization in \eqref{eq:3.2}, the entries $X_{n1},\ldots,X_{nn}$ are independent with distributions given by
\[
P(X_{n\nu} = 1) = p_{n\nu}\, , \qquad
P(X_{n\nu} = 0) = 1 - p_{n\nu}
\]
with numbers $p_{n\nu} \in [0,1]$. Now the aim is an asymptotic expansion for the probabilities
\[
p_{n,j} := P(S_n = j)
\]
as $n\to\infty$, uniformly in $j\in\Z$. In order to quote a relevant limit theorem we need some further notation and conditions. Suppose that $\mu_n := E(S_n)$ and $\sigma_n^2 = Var\, (S_n)$ are the expectation and the variance of $S_n$, respectively, for which we assume that
\begin{equation} \label{eq:3.3}
\liminf_{n\to\infty}\, \frac{\sigma_n^2}{n} > 0\, .
\end{equation}
The normalized cumulants of $S_n$ are defined by
\begin{equation} \label{3.4}
\lambda_{\nu,n} := \frac{n^{(\nu-2)/2}}{\sigma_n^\nu} \, \frac{1}{i^\nu} 
\left( \frac{d}{dt}\right)^\nu \log E(e^{itS_n})\Big|_{t=0}\, ,
\end{equation}
$n,\nu\in\N, \nu\geq 2$, where $E(e^{itS_n})$ is the characteristic function of $S_n$ and $\log$ is that branch of the logarithm on the cut plane $\C\setminus(-\infty,0]$ satisfying $\log 1 = 0$. Finally, we introduce the functions
\begin{equation} \label{eq:3.5}
q_{\nu,n} (x) := \frac{1}{\sqrt{2\pi}}~ e^{-x^2/2} 
\sum_{\mu_1 + 2\mu_2 + \ldots + \nu\mu_\nu = \nu} H_{\nu+2s} (x) \prod_{m=1}^\nu \, \frac{1}{\mu_m!}\,
\left( \frac{\lambda_{m+2,n}}{(m+2)!}\right)^{\mu_m},
\end{equation}
$x\in \R, n,\nu\in\N$, where $s = \mu_1 + \ldots + \mu_\nu$, and the modified Hermite polynomials are defined by
\begin{equation} \label{eq:3.6}
H_m (x) := (-1)^m e^{x^2/2} \left( \frac{d}{dx}\right)^m e^{-x^2/2},
\end{equation}
$x\in\R, m\in\N_0$. Now we can formulate our basic auxiliary result which we take from \cite[Lemma 3.1]{14}.
\begin{lemma} \label{lem3.1}
Assuming the above notations and the condition \eqref{eq:3.3}, then for every $k \geq 2$, we have
\begin{equation} \label{eq:3.7}
\sigma_n p(n,j) = \frac{1}{\sqrt{2\pi}}~ e^{-x^2/2} + \sum_{\nu=1}^{k-2}\, \frac{q_{\nu,n} (x)}{n^{\nu/2}} +
o \left( \frac{1}{n^{(k-2)/2}}\right),
\end{equation}
as $n\to\infty$, uniformly $j \in \Z$, where $x = (j - \mu_n)/\sigma_n$.
\end{lemma}
The reader who is not familiar with expansions of type \eqref{eq:3.7} is refered to the discussion accompanying Lemma \ref{lem3.1} in \cite{14} in general and to the comments following Theorem \ref{theo4.2} below.

\section{Central limit results for the Legendre-Stirling numbers}

In this final section we apply the preliminary results of the previous section with $A_n (s) = M_n (s)$ and
\begin{equation} \label{eq:4.1}
p(n,j) := \frac{(2j)! {n\brace j}_1}  {M_n (1)}\, ,
\end{equation}
where $M_n (s)$ is given by \eqref{eq:2.9}. This is possible in view of Theorem \ref{theo2.8}. First, we compute approximations for the expectation and the variance of the distribution in \eqref{eq:4.1}. To this end we use the notations introduced in Section 3. Looking at \eqref{eq:2.14} as in \cite{15} we introduce the number
\begin{equation} \label{eq:4.2}
\omega := 2 \log \frac{\sqrt{5} + 1}{2} = 0.9624\ldots
\end{equation}
being the unique positive solution of $2(\cosh w- 1) = 1$. Also in the sequel we will denote by $q\in (0,1)$ a constant which may be different at each occurrence. 
\begin{lemma} \label{lem4.1}
Suppose that the sequences $(a_n)$ and $(b_n)$ are given by
\begin{equation} \label{eq:4.3}
a_n = \frac{2n+1}{\sqrt{5} \omega} - \frac{1}{2}, \qquad
b_n = \left( \frac{1}{2} - \frac{\omega}{\sqrt{5}}\right) \left( \frac{2}{\omega}\right)^2 \frac{n}{5}\, ,
\end{equation}
then we have
\begin{equation} \label{eq:4.4}
\mu_n = a_n + {\mathcal{O}} \left( \frac{1}{n}\right), \qquad \sigma_n^2 = b_n + {\mathcal{O}} (1)\, ,
\end{equation}
as $n\to\infty$.
\end{lemma}
\begin{proof}
We use the well-known formulae
\begin{equation} \label{eq:4.5}
\mu_n = \frac{M_n' (1)}{M_n (1)}\, , \qquad
\sigma_n^2 = \frac{M_n'' (1)}{M_n (1)} + \frac{M_n'(1)}{M_n (1)} - \left( \frac{M_n' (1)}{M_n (1)}\right)^2
\end{equation}
from probability \cite{25}. Further we apply Lemmata \ref{lem2.3} - \ref{lem2.6} with $w = z = \omega$ giving asymptotic approximations for the derivatives in \eqref{eq:4.5}. More precisely, regarding Lemma \ref{lem2.3}, we put
\begin{equation} \label{eq:4.6}
f(w) := \frac{\sinh w}{\cosh w - 1} = \textrm{cotanh} \frac{w}{2}, \qquad
g(w) := \frac{1}{2(\cosh w - 1)} = \frac{1}{4 \sinh^2 \frac{w}{2}}
\end{equation}
for $w > 0$ and observe that
\begin{equation} \label{eq:4.7}
f'(w) = - 2g(w), \qquad g'(w) = -g(w) f(w).
\end{equation}
Now, \eqref{eq:2.14} can be written as
\begin{equation} \label{eq:4.8}
f(w) M_n \big( g(w)\big) = \sum_{m=-\infty}^\infty I_{n,n} (w + 2\pi im), \qquad w > 0.
\end{equation}
Differentiation, Lemma \ref{lem2.4} and the use of \eqref{eq:4.7} give
\begin{align}
\label{eq:4.9} 2g(w) M_n \big( g(w)\big)
& + g(w) f(w)^2 M_n' \big( g(w)\big) = \sum_{m=-\infty}^\infty I_{n+1,n} (w + 2\pi im), \\
2g(w) f(w) M_n \big( g(w)\big)
& + \Big\{ 6g(w)^2 f(w) + g(w) f(w)^3\Big\} \, M_n' \big( g(w)\big) \nonumber \\
\label{eq:4.10} & + g(w)^2 f(w)^3 M_n'' \big( g(w)\big) = \sum_{m=-\infty}^\infty I_{n+2,n} (w + 2\pi im)\, .
\end{align}
Since $f(w) = \sqrt{1+4g(w)}$ and $g(\omega) = 1,$ $f(\omega) = \sqrt{5},$ \eqref{eq:4.8} - \eqref{eq:4.11}, for $w = \omega$, imply
\begin{align}
\label{eq:4.11} \sqrt{5} M_n (1)
   & = \sum_{m=-\infty}^\infty I_{n,n} (\omega + 2\pi im), \\
\label{eq:4.12} 2 M_n (1) + 5 M_n' (1)
   & = \sum_{m=-\infty}^\infty I_{n+1,1} (\omega + 2\pi im), \\
\label{eq:4.13} 2\sqrt{5} M_n (1) + 11\sqrt{5} M_n' (1) + 5\sqrt{5} M_n'' (1)
   & = \sum_{m=-\infty}^\infty I_{n+2,n} (\omega + 2\pi im),
\end{align}
and further, by Lemma \ref{lem2.5},
\begin{align}
\label{eq:4.14}
   \frac{M_n' (1)}{M_n (1)} 
& = \frac{I_{n+1,n} (\omega)}{\sqrt{5} I_{n,n} (\omega)} \, 
   \Big( 1 + {\mathcal{O}} (q^n)\Big) - \frac{2}{5}\, , \\*[0.3cm]
\label{eq:4.15}
   \frac{M_n'' (1)}{M_n (1)} 
&= \frac{I_{n+2,2}(\omega)}{\sqrt{5} I_{n,n} (\omega)} \,
   \Big( 1 + {\mathcal{O}} (q^n)\Big) - \frac{11}{5} \, \frac{M_n' (1)}{M_n (1)} - \frac{2}{5}\, .
\end{align}
Now, tedious but straightforward and elementary calculations establish \eqref{eq:4.4}
\end{proof}
For later reference we note the following consequence of \eqref{eq:4.11}, \eqref{eq:2.18} and \eqref{eq:2.19} (also use
$\cosh \, \frac{\omega}{2} = \frac{\sqrt{5}}{2}$ and Stirling's formula)
\begin{align} 
M_n (1)
& = \frac{1}{\sqrt{5}}\, I_{n,n} (\omega) \big( 1 + {\mathcal{O}}(q^n)\big)  \nonumber \\*[0.3cm]
& = \frac{(n!)^2}{\sqrt{5\pi n}} \left(\frac{2}{\omega}\right)^{2n+1}\, \cosh \, \frac{\omega}{2} \,
    \left( 1 +{\mathcal{O}} \left( \frac{1}{n}\right)\right) \nonumber \\*[0.3cm]
\label{eq:4.16} & = \frac{(2n)!}{\omega^{2n+1}} \, \left( 1 + {\mathcal{O}} \, \left( \frac{1}{n}\right)\right)\, .
\end{align}
Now, our main result is stated in
\begin{theorem} \label{theo4.2}
Suppose that the sequences $(a_n), (b_n)$, and the number $\omega$ are given by \eqref{eq:4.3} and \eqref{eq:4.2} respectively, then we have
\begin{equation} \label{eq:4.18}
\frac{\sqrt{b_n} (2j)! \omega^{2n+1}}{(2n)!} \, {n\brace j}_1 =
\frac{1}{\sqrt{2\pi}}\, e^{-x^2/2} + o(1)\, ,
\end{equation}
as $n\to\infty$, uniformly in $j\in\Z$, where $x = (j-a_n)/\sqrt{b_n}$.
\end{theorem}
\begin{proof}
According to our preparations we apply Lemma \ref{lem3.1} with $k = 2$ to the probabilities in \eqref{eq:4.1}. This is permitted, since the central condition \eqref{eq:3.3} is satisfied in view of Lemma \ref{lem4.1}. Hence, we get
\begin{equation} \label{eq:4.19}
\frac{\sigma_n (2j)!}{M_n (1)} \, {n\brace j}_1 =
\frac{1}{\sqrt{2\pi}}\, e^{-y^2/2} + o(1)\, ,
\end{equation}
as $n\to\infty$, uniformly in $j\in\Z$, where $y = (j-\mu_n)/\sigma_n$ and $\mu_n, \sigma_n$ are taken from Lemma 
\ref{lem4.1}. Now, elementary calculus shows that in \eqref{eq:4.19} we may replace $y$ by $x$, given in \eqref{eq:4.18}, and finally the approximation \eqref{eq:4.16} leads to the main formula \eqref{eq:4.18}.
\end{proof}
At this stage some comments on Theorem \ref{theo4.2} are in order. We repeat that the error term in \eqref{eq:4.18} holds uniformly with respect to $j\in\Z$. However, the most valuable information is provided for $j$'s such that 
$x = (j-a_n)/\sqrt{b_n}$ is bounded. A similar statement holds for the general local central limit theorem in Lemma \ref{lem3.1} as well. If $j$ depends on $n$ such that $|x|$ tends to infinity sufficiently fast, then the right hand side is of order $o(1)$ only. According to these remarks we may expect a good approximation of ${n\brace j}_1$ by the quantity
\[
A(n,j) := \frac{(2n)!}{\sqrt{2\pi b_n} (2j)! \omega^{2n+1}}\, e^{-x^2/2},
\]
if $j$ is close to $a_n \sim 2n/\sqrt{5} \omega$. For illustration we choose $n = 1000$ and $j = 930$ which implies the relative comparison given by
\[
\frac{{n\brace j}_1}{A(n,j)} = 1.043849\ldots .
\]
We also mention that on the basis of the general Lemma \ref{lem3.1} we could improve the error term in \eqref{eq:4.18} which requires more terms for the expansion \eqref{eq:2.19}. However, here we do not perform the necessary calculations.

An immediate consequence of Theorem \ref{theo4.2} is the asymptotic normality of the numbers
$(2j)! {n\brace j}_1$. We omit detailed explanations, since they use routine arguments regarding approximations of integrals by means of Riemann sums.
\begin{theorem} \label{theo4.3}
Suppose that the sequences $(a_n), (b_n)$ and the number $\omega$ are given by \eqref{eq:4.3} and \eqref{eq:4.2} respectively, then for all $y\in\R$ we have
\begin{equation} \label{4.20}
\lim_{n\to\infty} \, \frac{\omega^{2n+1}}{(2n)!} \, \sum_{j\leq a_n + y\sqrt{b_n}} (2j)! 
{n\brace j}_1 =
\frac{1}{\sqrt{2\pi}} \int\limits_{-\infty}^y e^{-t^2/2} dt.
\end{equation}
\end{theorem}
The proofs for the asymptotics of the Legendre-Stirling numbers ${n\brace j}_1 $ in this paper and of the Chebyshev-Stirling numbers ${n\brace j}_{1/2}$ in \cite{15} to a large extent depend on the special representations  \eqref{eq:2.3}, \eqref{eq:2.4} and related analytic quantities. In the comments following Theorem \ref{theo2.9} above we briefly indicated why our approach cannot be applied to the Jacobi-Stirling numbers in general. However, for the special case $\gamma= 0$ we have
\begin{equation} \label{eq:4.21}
{n\brace j}_0 =
{n-1\brace j-1}_1,
\end{equation}
and, consequently, the results of this paper also hold for the numbers \({n\brace j}_0\). The identity \eqref{eq:4.21} either can be verified directly by using \eqref{eq:1.3} or by showing that both sequences in \eqref{eq:4.21} satisfy the same recurrence relation.

\section*{Acknowledgements} Thorsten Neuschel gratefully acknowledges support from KU Leuven research grant OT\slash12\slash073 and the Belgian Interuniversity Attraction Pole P07/18.






\end{document}